\documentclass[11pt]{article}
\usepackage{amsfonts,amssymb,amsmath,amsthm}
\usepackage{mathrsfs}
\usepackage{fullpage}
\usepackage{enumitem}
\usepackage{graphicx}
\usepackage{tikz}
\usepackage{hyperref}
\usepackage{blkarray}

\newtheorem{theorem}{Theorem}%[section]
\newtheorem{corollary}[theorem]{Corollary}

\newtheorem{proposition}[theorem]{Proposition}

\newtheorem{conjecture}[theorem]{Conjecture}

\newtheorem{statement}{Statement}%[section]

\newtheorem{question}[statement]{Question}

\newcommand{\floor}[1]{\left\lfloor#1\right\rfloor}
\newcommand{\ceil}[1]{\left\lceil#1\right\rceil}
\newcommand{\st}{\colon\,}

\newcommand{\diam}{\mathrm{diam}}
\newcommand{\comment}[1]{}

\setenumerate[0]{label=(\arabic*)}

\tikzstyle vertex=[circle, inner sep=0pt, minimum size=1ex, fill=black]

\begin{document}

\author{Kevin G. Milans, Michael C. Wigal}
\title{Online coloring a token graph}
\date{\today}

\newcommand{\rankarg}[1]{\chi_{#1}}
\newcommand{\krank}{\rankarg{k}}
\newcommand{\trank}{\rankarg{2}}
\newcommand{\infrank}{\rankarg{\infty}}
\newcommand{\chis}{\chi_{s}}
\newcommand{\F}{\mathbb{F}}
\newcommand{\dist}{\mathrm{dist}}
\newcommand{\A}{\mathcal{A}}
\newcommand{\G}{\mathcal{G}}

\newcommand{\ZZ}{\mathbb{Z}}
\newcommand{\PP}{\mathbb{P}}
\newcommand{\RR}{\mathbb{R}}
\newcommand{\GG}{\mathbb{G}}

\newcommand{\val}[2]{f(#1,#2)}
\newcommand{\rval}[2]{f^{*}(#1,#2)}

\maketitle
\begin{abstract}
We study a combinatorial coloring game between two players, Spoiler and Algorithm, who alternate turns.  First, Spoiler places a new token at a vertex in $G$, and Algorithm responds by assigning a color to the new token.  Algorithm must ensure that tokens on the same or adjacent vertices receive distinct colors.  Spoiler must ensure that the token graph (in which two tokens are adjacent if and only if their distance in $G$ is at most $1$) has chromatic number at most $w$.  Algorithm wants to minimize the number of colors used, and Spoiler wants to force as many colors as possible.  Let $\val{w}{G}$ be the minimum number of colors needed in an optimal Algorithm strategy.

A graph $G$ is \emph{online-perfect} if $\val{w}{G} = w$.  We give a forbidden induced subgraph characterization of the class of online-perfect graphs.  When $G$ is not online-perfect, determining $\val{w}{G}$ seems challenging; we establish $\val{w}{G}$ asymptotically for some (but not all) of the minimal graphs that are not online-perfect.  The game is motivated by a natural online coloring problem on the real line which remains open.  
\end{abstract}

\section{Introduction}

Our work is motivated by the following simple combinatorial game between two players, Spoiler (sometimes called Presenter) and Algorithm.  Spoiler moves first, selecting a point $x\in \RR$ (repetition allowed), with the restriction that for each open unit interval $I$, Spoiler may play in $I$ at most $w$ times.  Algorithm responds by assigning a color to $x$ such that selected points at distance less than $1$ are assigned distinct colors.  The value of the game, denoted $h(w)$, is the minimum number of colors that an optimal Algorithm strategy needs in the worst case.  The greedy algorithm shows that $h(w)\le 2w-1$, as follows.  If none of the colors in $\{1,\ldots,2w-1\}$ is available for a selected point $x$, then either $(x-1,x]$ or $[x,x+1)$ contains at least $w$ previously played points which, after a small shift, can be captured in an open unit interval $I$ containing $x$.  Spoiler is forbidden to play again in $I$, contradicting Spoiler's selection of $x$.  Despite its simplicity, the greedy algorithm gives the best known upper bound on $h(w)$.

From below, Bosek, Felsner, Kloch, Krawczyk, Matecki, and Micek~\cite{BosekSurvey} proved that $h(w)\ge \floor{\frac{3}{2}w}$.  In their work, the problem of bounding $h(w)$ arises as a special case of the online chain partition problem for partially ordered sets.  To obtain the lower bound on $h(w)$, they offer the following strategy for Spoiler, which has three stages.  Let $k=\floor{\frac{w}{2}}$.  In the first stage, Spoiler plays $k$ times at $0$.  Let $A$ be the set of colors that Algorithm uses.  In the second stage, Spoiler plays at most $2k$ times in the interval $(1,2)$ so that all points with new colors (outside $A$) are less than all points with old colors (in $A$); Spoiler accomplishes this by always playing at a point between the new and old colors.  Spoiler stops this stage once Algorithm uses $k$ new colors.  Let $B$ be the set of new colors, and let $y\in (1,2)$ be a point between the new and old colors.  Note that $(y-1,y)$ contains $1$ and therefore each color in $B$ appears on a selected point in $(y-1,y)$. Similarly $(y-2,y-1)$ contains $0$ and therefore each color in $A$ appears on a selected point in $(y-2,y-1)$.  In the third stage, Spoiler plays $w-k$ times at $y-1$.  Algorithm must use $w-k$ additional colors outside $A\cup B$.  In total, Algorithm uses $w+k$ colors, and $w+k=\floor{\frac{3}{2}w}$.  This strategy for Spoiler gives the best known lower bound on $h(w)$.

It would be interesting to know the asymptotics of $h(w)$.  In attempting to make progress, we studied a generalization of this problem to graphs.  An \emph{independent set} in a graph $G$ is a set of vertices that are pairwise nonadjacent.  The \emph{chromatic number} of $G$, denoted $\chi(G)$, is the minimum size of a partition of $V(G)$ into independent sets.

For our generalization, we use the \emph{token model}, in which both the graph $G$ and the width $w$ are known to both Algorithm and Spoiler.  Spoiler selects a vertex $u\in V(G)$ and plays a \emph{token} $x$ at $u$.  In the \emph{token graph}, tokens $x$ and $y$ are adjacent if and only if $x$ and $y$ are on the same or neighboring vertices in $G$.  Spoiler must play in such a way that the token graph has chromatic number at most $w$.  After Spoiler plays a token $x$, Algorithm must assign a color to $x$ so that Algorithm's coloring of the token graph is \emph{proper}, meaning that tokens on the same or adjacent vertices in $G$ must be assigned distinct colors.  The \emph{value} of the game, denoted $\val{w}{G}$, is the minimum number of colors that an optimal strategy for Algorithm needs for the game in the worst case.  The motivating problem arises when $G$ is the graph $\GG$ with $V(\GG)=\RR$ and $uv\in E(\GG)$ if and only if $|u-v| < 1$.  With the exception of a few explicitly defined infinite graphs like $\GG$, we generally assume that all graphs are finite.

In some cases, it is interesting to limit Spoiler so that Spoiler can play at most $1$ token at each vertex.  We call this the \emph{restricted token model} and use $\rval{w}{G}$ to denote the value of the restricted token game on $G$ of width $w$.  Using $\overline{G}$ for the graph complement of $G$, we note that the restricted token game on $\overline{\GG}$ is an online clustering problem: Spoiler presents distinct points in $\RR$ and Algorithm assigns colors so that the color classes, called \emph{clusters}, have diameter less than $1$.  (Usually, clusters are allowed to have diameter at most $1$, but standard perturbation and compactness arguments imply that these notions lead to equivalent games.)  The trivial bounds are $\floor{\frac{3}{2}w}\le \rval{w}{\overline{\GG}} \le 2w-1$.  The first nontrivial bounds are due to Epstein and van Stee~\cite{ES}, who proved $(\frac{8}{5}-o(1))w \le \rval{w}{\overline{\GG}} \le \frac{7}{4}w$.  Ehmsen and Larson~\cite{EL} improved the upper bound to $\rval{w}{\overline{\GG}} \le \frac{5}{3}w$ and Kawahara and Kobayashi~\cite{KK} improved the lower bound to $\rval{w}{\overline{\GG}} \ge (\frac{13}{8}-o(1))w$.

\subsection{Other Models}

Several models of online graph coloring have been studied.  In the \emph{classical model}, Spoiler presents the vertices of a graph in some order $v_1, \ldots, v_n$ and Algorithm must decide on a color for $v_i$ knowing only the subgraph induced by $\{v_1,\ldots,v_i\}$.  For an algorithm $\A$ and a graph family $\G$, the \emph{performance ratio} of $\A$ on $\G$, denoted $\rho_{\A}(n,\G)$, is the maximum, over all $n$-vertex graphs $G\in\G$ and orderings of $V(G)$, of the number of colors that $\A$ uses on $G$ divided by $\chi(G)$.  It is not difficult to show that even when $\G$ is the family of acyclic graphs, every algorithm $\A$ satisfies $\rho_{\A}(n,\G) \ge 1+\floor{\lg n}$, and so the performance ratio of an algorithm in the classical model typically depends on the number of vertices in the input graph.  In 1989, for the class $\G$ of all graphs, Lov\'asz, Saks, and Trotter~\cite{LST} proved that there exists an algorithm $\A$ such that $\rho_\A(n,\G) = (1+o(1))(2n/\log^* n)$, where $\log^* n$ is the least integer $k$ such that the $k$th iterated logarithm function $\log^{(k)}$ satisfies $\log^{(k)}(n) < 1$.  Also in 1989, M. Szegedy (unpublished) showed that every online algorithm $\A$ satisfies $\rho_\A(n,\G) \ge \Omega(n/(\log n)^2)$; for this and lower bounds on the performance of randomized algorithms, see Halld\'orsson and Szegedy~\cite{HS}.  In 1998, Kierstead~\cite{K1998} showed that there exists an algorithm $\A$ such that $\rho_\A(n,\G)\le O(n\log^{(3)}n/\log^{(2)}(n))$, giving the best known performance ratio for the class of all graphs in the classical model.  For an older but still useful survey on results in the classical model, see the survey paper by Kierstead~\cite{KiersteadSurvey}.

In addition to the token model that we study, there are at least two other \emph{known-graph models}.  Halld\'orsson~\cite{H2000} examined a model in which a host graph $G$ is known to both players.  Spoiler presents vertices to Algorithm, promising that the presented graph $H$ is a subgraph of $G$.  However, if $G$ contains many copies of $H$, then Algorithm does not know which of these copies is $H$.  Consider, for example, the case $G=C_6$.  If the first two vertices that Spoiler presents are nonadjacent, then the pair may be at distance $2$ in $G$ (in which case Algorithm prefers that they share a color) or at distance $3$ (in which case Algorithm prefers that they be colored differently).  The performance ratio of Algorithm is the number of colors that Algorithm uses divided by $\chi(G)$.  Halld\'orsson constructed for each $n$ an $n$-vertex graph $G$ such that every algorithm has performance ratio at least $\Omega(n/(\log n)^2)$ on $G$, obtaining essentially the same lower bound as Halld\'orsson--Szegedy against a stronger Algorithm.

A second known-graph model appears in a paper by Bartal, Fiat, and Leonardi~\cite{BFL} and an extended abstract by the same authors~\cite{BFLSTOC}.  Again, the host graph $G$ is known to both Algorithm and Spoiler.  In this model, Spoiler reveals the location of vertices in $G$.  The difficulty for Algorithm is that Spoiler may stop the game at any subgraph $H$ of $G$, and the performance ratio of Algorithm is the number of colors used divided by $\chi(H)$ (not $\chi(G)$).  This model can also be viewed as a slight variant of the restricted token model in which Algorithm does not know $w$.  As usual, bounds on the performance ratio are typically given in terms of $n$, where $n=|V(G)|$.  Bartal, Fiat, and Leonardi proved that there is an algorithm that achieves performance ratio $O(n^{1/2})$ for each $n$-vertex graph $G$ and that every algorithm has performance ratio $\Omega(n^{1-\log_5 4})$.  The lower bound applies even to randomized algorithms.  

\subsection{Our Results}

Since determining the asymptotics of $h(w)$, or equivalently $\val{w}{\GG}$, appears difficult, it is natural to study the token model in simpler settings.  A graph $G$ is \emph{online-perfect} if $\val{w}{G} = w$ for each $w$.  Our main result is a forbidden-subgraph characterization of the family of online-perfect graphs.  The minimal graphs that are not online-perfect are the odd cycles of length at least $5$, the $5$-cycle with $1$ or $2$ non-crossing chords, and the bull graph.  These graphs are perhaps the simplest candidates for host graphs where the token game becomes an interesting problem.  Of these, we obtain the asymptotic value of $\val{w}{G}$ when $G$ is an odd cycle or the $5$-cycle with $2$ non-crossing chords.  In particular, we show that if $n$ is odd and $n\ge 5$, then $\val{w}{C_n} = (\frac{n}{n-1}+o(1))w$ where $o(1)$ represents a function tending to $0$ as $w\to\infty$.  Also, $\val{w}{G} = \ceil{\frac{5}{4} w}$ when $w$ is even and $G$ is the $5$-cycle with $2$ non-crossing chords.  The cases that $G$ is the bull graph or that $G$ is a $5$-cycle with a single chord remain open.

\section{Characterization of online-perfect graphs}

Our first strategy for Algorithm shows that bipartite graphs are online-perfect.  A \emph{clique} in $G$ is a set of vertices that are pairwise adjacent, and the \emph{clique number} of a graph $G$, denoted $\omega(G)$, is the maximum size of a clique in $G$.  Since no independent set contains more than one vertex of a clique, we have $\chi(G)\ge \omega(G)$.

\begin{proposition}\label{prop:biop}
If $G$ is bipartite, then $\val{w}{G} = w$.
\end{proposition}
\begin{proof}
Let $G$ be an $(U,V)$-bigraph, and let $L$ be a list of $w$ colors.  Algorithm maintains the invariant that for each $u\in U$, the colors used for tokens on $u$ form a prefix of $L$.  Similarly, for each $v\in V$, the colors used for tokens on $v$ form a suffix of $L$.  When Spoiler plays a token $x$ at $u\in U$, Algorithm assigns to $x$ the first color in $L$ not already used for a token on $u$.  When Spoiler plays a token $y$ at $v\in V$, Algorithm assigns to $y$ the last color in $L$ not already used for a token on $v$.  We claim that Algorithm produces a proper coloring of the token graph.  Indeed, if $x$ and $y$ are assigned the same color but are on adjacent vertices $u\in U$ and $v\in V$, then $u$ and $v$ together contain more than $w$ tokens, implying that the token graph has clique number larger than $w$. 
\end{proof}

In a graph $G$, vertices $u$ and $v$ are \emph{twins} if $u$ and $v$ have the same neighbors in $V(G)-\{u,v\}$.  Given a graph $G$ and a vertex $u$ in $G$, we \emph{clone} $u$ by introducing a new vertex $u'$ that is a twin of $u$; in the new graph, $u$ and $u'$ may or may not be adjacent.  

\begin{proposition}\label{prop:twin}
If $G'$ is obtained from $G$ by cloning a vertex, then $\val{w}{G'} = \val{w}{G}$.
\end{proposition}
\begin{proof}
Suppose that $G'$ be obtained from $G$ by cloning $u$ to a new vertex $u'$.  Since $G$ is an induced subgraph of $G'$, we have $\val{w}{G'} \ge \val{w}{G}$.  To show $\val{w}{G'} \le \val{w}{G}$, we let $\A$ be an optimal strategy for Algorithm for $G$ and give an Algorithm strategy $\A'$ for $G'$.  The strategy depends on whether or not $uu'\in E(G)$.  

First, suppose $uu'\in E(G')$.  We have $\A'$ simulate $\A$ on a token game on $G$.  If Spoiler plays a token at a vertex $v\not\in\{u,u'\}$, then $\A'$ plays a token at $v$ in the simulation, and $\A'$ gives the real token the same color as $\A$ gives to the simulated token.  If Spoiler plays a token at $u$ or $u'$, then $\A'$ plays a token at $u$ in the simulation and again responds with the same color as $\A$.  Since $u$ and $u'$ are adjacent in $G'$, the token graph in the real game is the same as the token graph in the simulation.  Therefore the width of the real game and the width of the simulation are equal and we have $\val{w}{G'}\le \val{w}{G}$.

Suppose that $u$ and $u'$ are nonadjacent.  In this case, we modify the strategy slightly so that the number of tokens at $u$ in the simulation is equal to the maximum of the number of tokens at $u$ and $u'$ in the real game.  If Spoiler plays a token at $u$ or $u'$ which would increase this maximum, then $\A'$ plays a token at $u$ in the simulation and responds with the same color.  If Spoiler plays a token at $u$ when $u'$ already has at least as many tokens, then $\A'$ does not play any tokens in the simulation and responds with one of the colors appearing at $u'$ that does not appear at $u$.  The case that Spoiler plays a token at $u'$ when $u$ already has at least as many tokens is similar.  Again, the width of the real game and the width of the simulation are equal, and it follows that $\val{w}{G'} \le \val{w}{G}$.
\end{proof}

For graphs $G_1$ and $G_2$, we use $G_1 + G_2$ to denote the disjoint union of $G_1$ and $G_2$.  The \emph{join} of $G_1$ and $G_2$, denoted $G_1 \vee G_2$, is the graph obtained from disjoint copies of $G_1$ and $G_2$ with all vertices in the copy of $G_1$ adjacent to all vertices in the copy of $G_2$.  The following proposition, which we include for completeness, is an easy consequence of the theory of \emph{cographs}, the graphs that do not contain an induced copy of $P_4$.

\begin{proposition}\label{prop:P4-twins}
A graph $G$ is $P_4$-free if and only if $G$ is obtainable from $K_1$ by cloning vertices.
\end{proposition}
\begin{proof}
It is well known that if $G$ is $P_4$-free and $|V(G)|\ge 2$, then either $G = G_1 + G_2$ or $G= G_1 \vee G_2$ for some $P_4$-free graphs $G_1$ and $G_2$.  We obtain $G_1 + G_2$ by starting with a single vertex $u_1$, cloning to a nonadjacent vertex $u_2$, and inductively obtaining $G_1$ from $u_1$ and $G_2$ from $u_2$.  We obtain $G_1 \vee G_2$ similarly, except that we begin with adjacent vertices $u_1$ and $u_2$.

Conversely, cloning vertices preserves being $P_4$-free.  Indeed, it is easy to check that $P_4$ does not contain a pair of twins.  If $G$ is $P_4$-free but the graph $G'$ obtained from $G$ by cloning $u\in V(G)$ to a twin $u'\in V(G')$ produces a copy $P$ of $P_4$, then $u$ and $u'$ are both in $P$.  Since $u$ and $u'$ are twins in $G'$, they are also twins in $P$, a contradiction.
\end{proof}

When $G$ is an edge-transitive graph, we use $G^-$ to denote the graph obtained from $G$ by removing an edge.  Similarly, when the complement of $G$ is edge-transitive, we use $G^+$ to denote the graph obtained from $G$ by adding an edge.  The \emph{distance $k$-power} of $G$, denoted $G^k$, is the graph on $V(G)$ with $uv\in E(G^k)$ if and only if the distance between $u$ and $v$ in $G$ is at most $k$. 

\begin{theorem}\label{thm:online-perfect}
Let $G$ be a graph.  The following are equivalent.
\begin{enumerate}	
	\item\label{main:1} $G$ is online-perfect.
	\item\label{main:2} $\val{2}{G} = 2$.
	\item\label{main:3} $G$ does not contain any of the following as an induced subgraph:
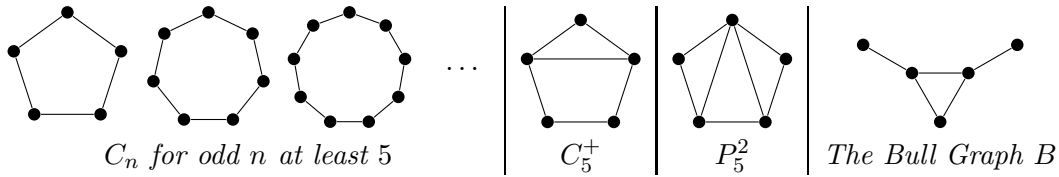
\begin{center}
\begin{tabular}{c|c|c|c}
\begin{tikzpicture}[scale=0.75]
	\begin{scope}[xshift=-2.5cm,rotate=18]
	\foreach \n in {0,...,4}
	{{
		\node[vertex] at (72*\n:1cm) (V\n) {} ;
	}}
	\draw (V0) -- (V1) -- (V2) -- (V3) -- (V4) -- (V0) ;
	\end{scope}

	\begin{scope}[rotate=38.57]
	\foreach \n in {0,...,6}
	{{
		\node[vertex] at (51.43*\n:1cm) (V\n) {} ;
	}}
	\draw (V0) -- (V1) -- (V2) -- (V3) -- (V4) -- (V5) -- (V6) -- (V0) ;
	\end{scope}

	\begin{scope}[xshift=2.5cm,rotate=10]
	\foreach \n in {0,...,8}
	{{
		\node[vertex] at (40*\n:1cm) (V\n) {} ;
	}}
	\draw (V0) -- (V1) -- (V2) -- (V3) -- (V4) -- (V5) -- (V6) -- (V7) -- (V8) -- (V0) ;
	\end{scope}

	\node at (4.5,0) {$\ldots$} ;
\end{tikzpicture} & 

\begin{tikzpicture}[scale=0.75]
	\begin{scope}[rotate=18]
	\foreach \n in {0,...,4}
	{{
		\node[vertex] at (72*\n:1cm) (V\n) {} ;
	}}
	\draw (V0) -- (V1) -- (V2) -- (V3) -- (V4) -- (V0) ;
	\draw (V0) -- (V2) ;
	\end{scope}
\end{tikzpicture} & 

\begin{tikzpicture}[scale=0.75]
	\begin{scope}[rotate=18]
	\foreach \n in {0,...,4}
	{{
		\node[vertex] at (72*\n:1cm) (V\n) {} ;
	}}
	\draw (V0) -- (V1) -- (V2) -- (V3) -- (V4) -- (V0) ;
	\draw (V3) -- (V1) -- (V4);
	\end{scope}
\end{tikzpicture} & 

\begin{tikzpicture}[scale=0.75]
	\begin{scope}[every node/.style={vertex}]
		\path (0,0) node (U) {} ++(60:1cm) node (V) {} ++(30:1cm) node (W) {} ;
		\path (U) ++(120:1cm) node (X) {} ++(150:1cm) node (Y) {} ;
	\end{scope}

	\draw (U) -- (V) -- (X) -- (U) ;
	\draw (X) -- (Y) ;
	\draw (V) -- (W) ;
\end{tikzpicture} \\

$C_n$ for odd $n$ at least $5$ &
$C_5^+$ &
$P_5^2$ &
The Bull Graph $B$
\end{tabular}
\end{center}
	\item\label{main:4} $G$ is obtainable from a bipartite graph by cloning vertices.
\end{enumerate}
\end{theorem}
\begin{proof}
It is obvious that \ref{main:1} implies \ref{main:2}.  To show that \ref{main:2} implies \ref{main:3}, we give strategies for Spoiler that force $3$ colors in a game of width $2$.  Let $u$ and $v$ be vertices in a graph $G$ that has induced $uv$-paths of both parities.  Spoiler plays a token $x$ at $u$ and a token $y$ at $v$.  If Algorithm assigns the same color to $x$ and $y$, then Spoiler forces two more colors by playing a token at each of the internal vertices on an induced $uv$-path of odd length.  If Algorithm assigns distinct colors to $x$ and $y$, then Spoiler forces a third color by playing a token at each of the internal vertices on an induced $uv$-path of even length.  In both cases, the token graph is a path, and so the game has width $2$.  Except for the bull graph $B$, each graph in \ref{main:3} has vertices $u$ and $v$ and induced $uv$-paths of both parities.  A different Spoiler strategy is needed for $B$.  Let $v_1, \ldots, v_5$ be the vertices of $B$ along its spanning path.  Spoiler first plays a token $x$ at $v_1$ and a token $y$ at $v_5$.  If Algorithm assigns the same color to $x$ and $y$, then Spoiler plays a token at $v_2$ and another at $v_4$.  This forces two new colors, and the token graph is $P_4$.  Otherwise, Algorithm assigns distinct colors to $x$ and $y$.  Spoiler plays a token $z$ at $v_3$; without loss of generality, we may assume that $x$ and $z$ have distinct colors.  Spoiler forces a third color by playing a token at $v_2$.  With one token on each vertex in $\{v_1, v_2, v_3, v_5\}$, the components of the token graph are $P_3$ and $P_1$, and again the token graph is bipartite.  It follows that if $\val{2}{G} = 2$, then $G$ does not contain an induced copy of any of the graphs listed in \ref{main:3}.

We show that \ref{main:3} implies \ref{main:4} by contradiction; let $G$ be a minimum counter-example.  If $G$ contains twins $u$ and $u'$, then $G-u'$ is a smaller graph avoiding induced copies of all graphs in \ref{main:3}, implying that $G-u'$ is obtainable from a bipartite graph by cloning.  We could then obtain $G$ from $G-u'$ by cloning $u$.  It follows that $G$ does not contain twins.  We consider two cases.  First, suppose that $\diam(G)\le 2$.  It must be that $G$ contains an induced copy $P$ of $P_4$, or else Proposition~\ref{prop:P4-twins} implies that $G$ is obtainable by cloning from the bipartite graph $K_1$.  Let $u_1, \ldots, u_4$ be the vertices of $P$.  Since $P$ is induced and $\diam(G)\le 2$, it follows that $u_1$ and $u_4$ have a common neighbor $v$ that is not on $P$.  Hence $v$ completes a $5$-cycle with $P$ in which all chords are incident to $v$, and so $V(P)\cup\{v\}$ induces one of $\{C_5, C_5^+, P_5^2\}$, all of which are forbidden.

We may assume that $\diam(G)\ge 3$.  Among all vertices in $G$ that are at distance $3$ from some vertex, select a vertex $u$, favoring vertices whose neighborhoods are independent sets. For $k\ge 0$, let $G_k$ be the subgraph of $G$ induced by vertices at distance $k$ from $u$.  We claim that if $H$ is a component of $G_k$ and $v_1\in V(G_{k-1})$, then $v_1$ is adjacent to all vertices in $H$ or none of them.  Indeed, if this is not the case, then let $k$ be the least integer where this fails.  Since $v_1$ has a neighbor and a non-neighbor in $H$ and $H$ is connected, there exists $v_0w_0 \in E(H)$ such that $v_1v_0 \in E(G)$ but $v_1w_0\not\in V(G)$.  Let $w_1\in V(G_{k-1})$ be a neighbor of $w_0$, and let $t$ be the least integer such that $G$ has paths $v_0v_1\cdots v_tx$ and $w_0w_1\cdots w_tx$ with $v_j,w_j\in V(G_{k-j})$ and $x\in V(G_{k-(t+1)})$.  By the minimality of $t$, we have that $v_iw_{i+1},w_iv_{i+1}\not\in E(G)$ for $1\le j < t$.  By the minimality of $k$, we have that $v_iw_i\not\in E(G)$ for $1\le i < t$.  It follows that $v_0v_1\cdots v_txw_t w_{t-1}\cdots w_0$ is a $(2t+3)$-cycle $C$ in which the chords are a subset of $\{v_0w_1,v_tw_t\}$.  If $t\ge 2$, then either $C$ is an induced odd cycle of length at least $7$, or either chord completes an induced copy of the bull.  If $t=1$, then $C$ is a $5$-cycle in which all chords are incident to $w_1$, and so $V(C)$ induces a graph in $\{C_5,C_5^+,P_5^2\}$.

Next, we claim that if $H$ is a component of $G_k$ with $k\ge 2$ and $x\in V(G_{k+1})$, then $x$ is adjacent to all vertices in $H$ or none of them.  Otherwise, we again obtain $vw\in E(H)$ such that $xv \in E(G)$ but $xw \not\in E(G)$.  Let $y\in G_{k-1}$ be a neighbor of $v$; also $yw\in E(H)$.  Let $z\in V(G_{k-2})$ be a neighbor of $y$, and observe that $\{v,w,x,y,z\}$ induces a copy of the bull.

We claim that $G_k$ has no edges for $k\ge 2$.  Let $H$ be a component of $G_k$, and let $x\in V(G_{k-1})$ contain $H$ in its neighborhood.  Note that $H$ is $P_4$-free, or else an induced copy of $P_4$ in $H$ together with $x$ yields an induced copy of $P_5^2$.  It follows from Proposition~\ref{prop:P4-twins} that either $H=K_1$ or $H$ contains a pair of twins.  Since $k\ge 2$, a pair of twins in $H$ would also be a pair of twins in $G$, and so $H=K_1$ as claimed.  

Note that $G_1$ also has no edges.  Indeed, if $z$ is a vertex at maximum distance $k$ from $u$, then $k\ge 3$ and $z$ is an isolated vertex in $G_k$.  It follows that all neighbors of $z$ are isolated vertices in $G_{k-1}$ and therefore $N(z)$ is an independent set.  If $N(u)$ were not independent, then we would prefer $z$ to $u$ in our initial selection of $u$.  It follows that $N(u)$ is an independent set, and $V(G_1) = N(u)$.  Since each $G_k$ has no edges, it follows that $G$ is a bipartite graph, contradicting that $G$ is a counter-example.  Therefore \ref{main:3} implies \ref{main:4}.

It remains to show that \ref{main:4} implies \ref{main:1}.  By Proposition~\ref{prop:biop}, every bipartite graph is online-perfect.  By Proposition~\ref{prop:twin}, cloning vertices preserves online-perfection.  
\end{proof}

A graph $G$ is \emph{perfect} if $\chi(H) = \omega(H)$ for each induced subgraph $H$ of $G$.  The Strong Perfect Graph Theorem of Chudnovsky, Robertson, Seymour, and Thomas~\cite{SPGT} states that $G$ is perfect if and only if $G$ does not contain an odd hole or an odd antihole.  A \emph{hole} is an induced cycle of length at least $4$, and an \emph{antihole} is an induced subgraph whose complement is a hole.  Note that a set of $5$ consecutive vertices on a cycle of length larger than $5$ induces a copy of $P_5$, and $\overline{P_5} = C_5^+$.  It follows that an antihole on more than $5$ vertices contains an induced copy of $C_5^+$.  By \ref{main:3}, odd holes and odd antiholes are forbidden in online-perfect graphs, and so every online-perfect graph is perfect.

It also follows from \ref{main:4} and Proposition~\ref{prop:P4-twins} that $G$ is online-perfect if and only if $G$ is obtainable from a bipartite graph by replacing each vertex with a copy of a $P_4$-free graph.

\section{Minimal graphs that are not online-perfect}

When $G$ is not online-perfect, it contains an induced copy of one of the obstructions from Theorem~\ref{thm:online-perfect}.  It is natural to study the value of the online-coloring game for these graphs; although these are the simplest graphs that are not online-perfect, we are only able to establish the asymptotics for the odd cycles and $P_5^2$.  We begin with a strategy for Algorithm that generalizes Proposition~\ref{prop:biop}.

A \emph{$(p,q)$-coloring} of a graph $G$ is an assignment $\phi$ that maps each vertex in $G$ to a set of $q$ colors from a universe of $p$ colors such that $\phi(u)$ and $\phi(v)$ are disjoint when $uv\in E(G)$.  Usually, we use $[p]$ for the universe of colors and write $\phi\st V(G)\to\binom{[p]}{q}$, where $\binom{[p]}{q}$ denotes the family of subsets of $[p]$ of size $q$.  The \emph{fractional chromatic number} of $G$, denoted $\chi_f(G)$, is the infimum of $p/q$ over all pairs $(p,q)$ such that $G$ has a $(p,q)$-coloring.  There is an equivalent alternative formulation of the fractional chromatic number as the optimum value of a linear program involving the independent set polytope of $G$, and it follows that each graph $G$ has a $(p,q)$-coloring such that $\chi_f(G) = p/q$.

\begin{proposition}[Fractional Coloring Strategy]
If $G$ has a $(p,q)$-coloring, then $\val{w}{G}\le p\ceil{\frac{w}{2q}}$.  Consequently, $\val{w}{G}\le (\frac{\chi_f(G)}{2} + o(1))w$.
\end{proposition}
\begin{proof}
Let $\phi\st V(G)\to\binom{[p]}{q}$ be a $(p,q)$-coloring, where $\chi_f(G)=p/q$.  For a list $L$, let $L^r$ denote the list obtained by reversing $L$.  Also, let $L_1\cdot L_2$ be the concatenation of lists $L_1$ and $L_2$, and when $I$ is a set of integer indices, let $\prod_{i\in I} L_i$ denote the concatenation of lists where entries in $L_i$ come before $L_j$ if and only if $i<j$. 

Let $t=\ceil{w/(2q)}$ and let $S_1, \ldots, S_p$ be disjoint lists of $t$ colors.  For each vertex $u$, the \emph{forward list} at $u$, denoted $F_u$, is $[\prod_{j\in A} S_j]$ and the \emph{reverse list} at $u$, denoted $R_u$, is $[\prod_{j\in B} S_j]^r$, where $A=\phi(u)$ and $B=[p]-\phi(u)$.  Let $L_u = F_u \cdot R_u$.  Algorithm maintains the invariant that, for each vertex $u$, the tokens played at $u$ are colored with a prefix of $L_u$.  When Spoiler plays a token $x$ at $u$, Algorithm assigns to $x$ the first color in $L_u$ that is not already used for a token at $u$.  Provided that Algorithm produces a proper coloring of the token graph, we have $\val{w}{G} \le tp = \ceil{\frac{w}{2q}}p < (\frac{w}{2q} + 1)p = (\frac{p}{2q} + \frac{p}{w})w = (\chi_f(G)/2 + o(1))w$.

Suppose for a contradiction that a token $x$ is played at $u$ but its designated color $\alpha$ already appears on a token at a neighbor $v$.  Since $\phi(u)$ and $\phi(v)$ are disjoint, the colors in $F_u$ appear in $R_{v}$ in reverse order, and colors in $F_{v}$ appear in $R_u$ in reverse order.  We claim that all colors in $F_u\cup F_{v}$ appear on tokens at $u$ or $v$.  Note that all colors preceding $\alpha$ in $F_u \cdot R_u$ appear on tokens at $u$.  If $\alpha$ is in $F_u$, then colors in $F_u$ following $\alpha$ appear before $\alpha$ in $R_{v}$.  By the prefix invariant at $v$, all colors preceding $\alpha$ in $F_{v} \cdot R_{v}$ appear on tokens at $v$; these colors include $F_{v}$ as well as those in $F_u$ following $\alpha$.  Otherwise, $\alpha \in R_u$, and so every color in $F_u$ appears on tokens at $u$.  If $\alpha \in R_{v}$, then every color in $F_{v}$ appears on tokens at $v$.  If $\alpha \in F_{v}$, then colors in $F_{v}$ preceding $\alpha$ appear at $v$ and colors in $F_{v}$ following $\alpha$ appear before $\alpha$ in $R_u$ and therefore are used on tokens at $u$.

In all cases, since $|F_u| = |F_{v}| = qt$, it follows that at least $2qt$ tokens other than $x$ have been played at $u$ and $v$.  These tokens form a clique in the token graph, and so $w > 2qt$, contradicting our choice of $t$.
\end{proof}

For a graph $G$ and a set of vertices $S\subseteq G$, we use $G[S]$ to denote the subgraph of $G$ induced by $S$.  

\begin{proposition}[Spoiler Strategy]
Let $U \subseteq V(G)$ where $U = \{u_1,...,u_t\}$, and suppose that
\begin{enumerate}
\item there are vertices $x$ and $y$ such that $u_1xyu_t$ is a path and $G[\{u_1,x,y,u_t\}]$ is bipartite, and
\item there is a common neighbor $z_i$ of $u_i$ and $u_{i+1}$ such that $G[U \cup \{z_i\}]$ is bipartite for $1\le i < t$.
\end{enumerate}
If $w$ is even, then $\val{w}{G} \geq (1 + \frac{1}{2t})w$.\\
\end{proposition}
\begin{proof}
With $k = w/2$, Spoiler plays $k$ tokens at $u_1$ and $u_t$.  Let $S_i$ be the set of colors at $u_i$.  If $|S_1 \cap S_t| \geq k/t$ then Spoiler plays $k$ tokens at each of $\{x,y\}$. This forces $2k + k/t$ colors.  Since the subgraph of $G$ induced by $\{u_1,x,y,u_t\}$ is bipartite, the chromatic number of the token graph is at most $2k$.
 
Otherwise $|S_1 \cap S_t| < k/t$ and $|S_1 - S_t| > k - k/t$.  Spoiler plays $k$ tokens at each $u_i$ for $1 < i < t$.  Since $S_1 - S_t$ is contained in $\bigcup_{i<t} S_i - S_{i+1}$, it follows that for some $i$, we have $|S_i - S_{i+1}| \geq \frac{|S_1 - S_t|}{(t-1)} > \frac{k}{t}$. Spoiler plays $k$ tokens at $z_i$.  Algorithm uses disjoint sets of $k$ colors at $z_i$ and $u_{i+1}$, plus at least $k/t$ additional colors at $u_i$; this forces at least $2k+k/t$ colors in total.  Since the subgraph of $G$ induced by $U\cup\{z_i\}$ is bipartite, the token graph has chromatic number at most $2k$ in this case also.
\end{proof}
 
\begin{corollary}[Odd Holes]\label{cor:oddhole}
Let $n$ be odd and at least $5$.  We have $\frac{n}{n-1}\cdot 2\floor{\frac{w}{2}} \le \val{w}{C_n} \le n\ceil{\frac{w}{n-1}}$ and $\val{w}{C_n} = \frac{n}{n-1} w$ when $n-1$ divides $w$.  Therefore $\val{w}{C_n} = (\frac{n}{n-1} + o(1))w$, where the $o(1)$ term tends to $0$ for each fixed $n$ as $w\to\infty$.
\end{corollary}
\begin{proof}
Let $w' = 2\floor{\frac{w}{2}}$, so that $w'$ is the largest even integer at most $w$.  With $t=(n-1)/2$ and $u_1, \ldots, u_t$ spaced along $C_n$ so that $u_i$ and $u_{i+1}$ are at distance $2$ for $1\le i < t$, the Spoiler Strategy gives $\val{w}{C_n} \ge \val{w'}{C_n} \ge (1 + \frac{1}{2t})w' = \frac{n}{n-1}\cdot 2\floor{\frac{w}{2}}$.  Since $C_n$ has an $(n,(n-1)/2)$-coloring, the Algorithm Strategy gives $\val{w}{C_n} \le n\ceil{\frac{w}{n-1}}$.  When $n$ is odd and $n-1$ divides $w$, the bounds reduce to $\val{w}{C_n} = \frac{n}{n-1} w$.
\end{proof}

\begin{corollary}\label{cor:C5p}
We have $\frac{5}{2}\floor{\frac{w}{2}} \le \val{w}{C_5^+} \le  3\ceil{\frac{w}{2}}$.
\end{corollary}
\begin{proof}
Let $w' = 2\floor{\frac{w}{2}}$, so that $w'$ is the largest even integer at most $w$. Let $u_1$ and $u_2$ be two nonadjacent vertices of degree two belonging to $C_5^+$.  The Spoiler Strategy gives $\val{w'}{C_5^+} \ge \val{w'}{C_5^+} \ge \frac{5}{4}w' = \frac{5}{2}\floor{\frac{w}{2}}$. Since $C_5^+$ has a $(3,1)$-coloring, the Algorithm Strategy gives $\val{w}{C_5^+} \le 3\ceil{\frac{w}{2}}$.
\end{proof}

\begin{corollary}\label{cor:bull}
We have $\frac{7}{3}\floor{\frac{w}{2}} \le \val{w}{B} \le 3\ceil{\frac{w}{2}}$.
\end{corollary}
\begin{proof}
Let $w' = 2\floor{\frac{w}{2}}$, so that $w'$ is the largest even integer at most $w$. In the bull $B$, let $u_1$ and $u_3$ be vertices of degree 1 and let $u_2$ be the vertex of degree 2.  The Spoiler Strategy gives $\val{w}{B} \ge \val{w'}{B}\ge \frac{7}{6}w' = \frac{7}{3} \floor{\frac{w}{2}}$. Since $B$ has a $(3,1)$-coloring, the Algorithm Strategy gives $\val{w}{B} \le 3\ceil{\frac{w}{2}}$.
\end{proof}

\section{Asymptotics for $P_5^2$}
 
Applying the argument of Corollary~\ref{cor:C5p} to $P_5^2$ shows that $\val{w}{P_5^2}$ is bounded by $(1-o(1))\frac{5}{4}w$ from below and $(1+o(1))\frac{3}{2}w$ from above.  In this section, we improve Algorithm's strategy to obtain $\val{w}{P_5^2}$ exactly when $w$ is even.

\begin{theorem}\label{thm:P52}
We have $\frac{5}{2}\floor{\frac{w}{2}} \le \val{w}{P_5^2} \le \frac{5w+2}{4}$ for each $w$.  Consequently, when $w$ is even, we have $\val{w}{P_5^2} = \ceil{\frac{5}{4}w}$.
\end{theorem}

\begin{proof}
The lower bound follows the same argument as in Corollary~\ref{cor:C5p}.  For the upper bound, we give a strategy for Algorithm.  It is convenient to label the vertices of $P_5^2$ according to the structure of the complementary graph.  The complement of $P_5^2$ is the union of an isolated vertex $v_0$ and a path $v_1v_2v_3v_4$.  Each color class that Algorithm uses appears on vertices forming a nonempty independent set in $P_5^2$; these are the $5$ singletons plus the pairs $v_1v_2$, $v_2v_3$, and $v_3v_4$.  For each nonempty set $I\subseteq \{0,\ldots,4\}$, we let $y_I$ be the number of colors assigned by algorithm that appear on vertices in $\{v_i \st i\in I\}$ and no other vertex.  We view the $y_I$ as variables whose values change throughout the game, and we define the vector $y=(y_0, y_1, y_{12}, y_2, y_{23}, y_3, y_{34}, y_4)$.

We are now able to describe Algorithm's strategy.  If Spoiler plays a token at $v_0$, then Algorithm responds by assigning the token a new color.  (Since $v_0$ is dominating in $P_5^2$, Algorithm has no choice here.)  The pairs $v_1v_2$ and $v_3v_4$ are \emph{greedy}; that is, if Spoiler plays a token at a vertex in $\{v_1,\ldots,v_4\}$, then Algorithm first attempts to extend a color class which appears only on its greedy mate to accommodate the new token.  Otherwise, Algorithm assigns the token a new color unless this would result in $\min\{y_2,y_3\}> (w+1)/4$.  In this case, Spoiler has played a token at $v_i\in \{v_2,v_3\}$; let $j$ index the other vertex in $\{v_2,v_3\}$.  Algorithm assigns the new token at $v_i$ a color which appears only at $v_j$.  

Observe that Algorithm maintains the following invariants.  Since $v_1v_2$ and $v_3v_4$ are greedy pairs, we have $\min\{y_1,y_2\} = \min\{y_3,y_4\} = 0$.  We also have $\min\{y_2,y_3\} \le (w+1)/4$.  Also, since $v_0v_1v_3$, $v_0v_1v_4$, and $v_0v_2v_4$ are triangles in $P_5^2$, the width condition implies that each of these triples has at most $w$ tokens.  This gives the following three bounds, which we present as a matrix inequality.

\[
\begin{blockarray}{*{9}{c}}
& y_0 & y_1 & y_{12} & y_2 & y_{23} & y_3 & y_{34} & y_4\\
\begin{block}{c(*{8}{c})}
v_0v_1v_3 & 1 & 1 & 1 &   & 1 & 1 & 1 &   \\
v_0v_1v_4 & 1 & 1 & 1 &   &   &   & 1 & 1 \\
v_0v_2v_4 & 1 &   & 1 & 1 & 1 &   & 1 & 1 \\
\end{block}
\end{blockarray}
~~y~~\le~~% 
\begin{blockarray}{c}
\\
\begin{block}{(c)}
w \\ w \\ w \\
\end{block}
\end{blockarray}
\]

Suppose that the game is played to its conclusion.  We bound the total number of colors $S$ used by Algorithm by examining $4$ cases, according to which variables in the pairs $\{y_1,y_2\}$ and $\{y_3,y_4\}$ are zero.  Let $S_1$, $S_2$, and $S_3$ be the sums of the terms on the left hand sides of the inequalities in rows 1, 2, and 3, respectively.  If $y_2 = y_4 = 0$, then we have $S = S_1 + y_2 + y_4 \le w$.  Similarly, if $y_1=y_3=0$, then $S= S_3 + y_1 + y_3 \le w$.  If $y_1 = y_4 = 0$, then we have $S=S_1+y_2 = S_3+y_3$ and hence $S\le w+\min\{y_2,y_3\} \le (5w+1)/4$.

Otherwise, $y_2 = y_3 = 0$.  If also $y_{23} = 0$, then $S = S_2 + y_2 + y_{23} + y_3 \le w$.  Hence we may assume that $y_{23} > 0$, implying that at some point Algorithm extends a color which appears on only one of $v_2,v_3$ so that it appears on both.  Consider the last such time that Algorithm increases $y_{23}$.  Before the corresponding token is played, we have $\max\{y_2,y_3\} > (w+1)/4$ and $\min\{y_2,y_3\} + 1 > (w+1)/4$.  After this token is played and algorithm assigns a color increasing $y_{23}$, we have that $\min\{y_2,y_3\} > (w+1)/4 - 1$ and so $\min\{y_2,y_3\} \ge (w+2)/4 -1$.  However, as this is the last time that $y_{23}$ is increased and yet $y_2 = y_3 = 0$ at the end of the game, it must be that $y_{12}$ and $y_{34}$ are each incremented at least $(w+2)/4 - 1$ times.  Therefore, on termination, we have $y_{12} + y_{34} \ge (w+2)/2 - 2$.  With respect to the values at termination, we compute
\begin{align*}
3S &= (S_1 + y_2 + y_4) + (S_2 + y_2 + y_{23} + y_3) + (S_3 + y_1 + y_3) \\
&\le 3w + y_1 + y_4 + y_{23} \\
&= 3w + S - (y_0 + y_{12} + y_{34})\\
&\le 3w + S - \frac{w-2}{2}
\end{align*}
and it follows that $S\le (5w+2)/4$.
\end{proof}

\section{Open Problems}

There are many open problems.  Perhaps the two most compelling problems are to determine the asymptotics of $\val{w}{\GG}$ and $\rval{w}{\overline{\GG}}$.  There are natural extensions of both problems to higher dimensions, where $\GG_{d,p}$ is the graph on $\RR^d$ with $uv\in E(\GG_{d,p})$ if and only if the $p$-norm distance between $u$ and $v$ is less than $1$.  Chan and Zarrabi-Zadeh~\cite{CZZ} noted that if $\rval{w}{\overline{\GG}} \le cw$ for some constant $c$, then $\rval{w}{\overline{\GG_{d,\infty}}} \le c2^{d-1}w$.  We are not aware of any other work on the higher-dimensional analogues.

It would be nice to obtain $\val{w}{G}$ asymptotically for the remaining minimally non-online-perfect graphs, when $G=C_5^+$ or $G$ is the bull graph.  We believe the following may be approachable.

\begin{conjecture} 
$\val{w}{C_5^+}= (\frac{5}{4} + o(1))w$.
\end{conjecture}

By constructing large graphs that simulate the flexibility of Spoiler in the classical model, one can show that for each $k$, there exists a graph $G_k$ such that $\val{2}{G_k} \ge k$.  However, the following question remains unresolved.

\begin{question}
Is there a constant $c$ such that $\val{w}{G}\le (c + o(1))w$ for each (finite) graph $G$?
\end{question}

We know only that $c$ must be at least $5/4$.

\section*{Acknowledgments}
We thank Csaba Biro for preliminary discussions on the value of the game on powers of infinite paths.

\bigskip

\bibliographystyle{plain}
\bibliography{citations}

\end{document}